\documentclass[reqno,11pt,centertags]{amsart}
\usepackage{geometry}
\usepackage{amsmath,amsthm,amscd,amssymb,latexsym,upref,stmaryrd}
\usepackage{graphicx}
\usepackage{hyperref}
\usepackage{enumitem}
\usepackage{xcolor}
\newcommand*{\mailto}[1]{\href{mailto:#1}{\nolinkurl{#1}}}
\usepackage[numbers,sort&compress]{natbib}

\newcommand{\beq}{\begin{equation}}
	\newcommand{\eeq}{\end{equation}}
\newcommand{\ba}{\begin{align}}
	\newcommand{\ea}{\end{align}}

 \allowdisplaybreaks
\numberwithin{equation}{section}

\newtheorem{theorem}{Theorem}[section]
\newtheorem{lemma}[theorem]{Lemma}

\theoremstyle{definition}

\begin{document}
	
	
	\title[Ambarzumyan-type theorem]
	{Ambarzumyan-type theorem for vectorial Sturm-Liouville operator with impulses}
	
	\author[F.~Wang]{Feng Wang}
	\address{School of Mathematics and Statistics, Nanjing University of
		Science and Technology, Nanjing, 210094, Jiangsu, China}
	\email{\mailto{wangfengmath@njust.edu.cn}}
	
	\author[C.~F.~Yang]{CHUAN-FU Yang}
	\address{Department of Mathematics, School of Mathematics and Statistics, Nanjing University of
		Science and Technology, Nanjing, 210094, Jiangsu, People's
		Republic of China}
	\email{\mailto{chuanfuyang@njust.edu.cn}}

	\date{\today}
	
	\begin{abstract}
{We consider the vector-impulsive Sturm-Liouville problem with Neumann conditions. The Ambarzumyan$^{\textbf{,}}$s theorem for the problem is proved, which states that if the eigenvalues of the problem coincide with those of the zero potential, then the potential is zero.}
	\end{abstract}
	
	\maketitle
	
\section{introduction}
	 It is well known that the paper \cite{Am} of  Ambarzumyan may be thought of as the first paper in the theory of inverse spectral problems associated with Sturm-Liouville operators. \cite{Am} stated the following theorem:

\emph{If} $q$ \emph{is a smooth real-valued function, and} $\{n^{2}:n=0,1,2,\ldots\}$ \emph{is the set of eigenvalues of the boundary value problem}
\begin{align}
-y''+q(x)y=\lambda y,\; x\in(0,\pi),\quad y'(0)=y'(\pi)=0, \nonumber
\end{align}
\emph{then} $q(x)=0$ on $[0,\pi]$.

This theorem is called Ambarzumyan$^{\textbf{,}}$s theorem, and has been generalized in many directions. Without a claim to completeness we mention here the papers [2-9, 11-13], etc. In particular, the most recent paper \cite{ZY} extended the Ambarzumyan$^{\textbf{,}}$s theorem for the classical Sturm-Liouville operator to the scalar impulsive Sturm-Liouville operator with Neumann conditions. Our interest is to extend the result of the paper \cite{ZY} to the vector case.

In this paper we consider the following vectorial Sturm-Liouville operator $L(Q)$ with impulses
\begin{align}
lY:=-Y''+Q(x)Y=\lambda\rho(x)Y, \quad x\in \left(0,\frac{\pi}{2}\right)\cup\left(\frac{\pi}{2},\pi\right), \label{intro1}
\end{align}
with the jump conditions
\begin{align}
Y\left(\frac{\pi}{2}+0\right)=aY\left(\frac{\pi}{2}-0\right), \quad Y'\left(\frac{\pi}{2}+0\right)=a^{-1}Y'\left(\frac{\pi}{2}-0\right), \!\!\!\!\!\!\!\!\! \label{intro2}
\end{align}
and Neumann conditions
\begin{align}
Y'(0)=Y'(\pi)=0, \label{intro3}
\end{align}
where 
\begin{align}
\rho(x)=\begin{cases}
          1,\quad\:\: 0<x<\frac{\pi}{2},\\
          \alpha^{2},\quad\frac{\pi}{2}<x<\pi,
        \end{cases}
0<\alpha<1, \nonumber
\end{align}
$a>0$ and $Q$ is an $N\times N$ real symmetric matrix-valued function and $Q\in L^{2}_{N\times N}(0,\pi)$, $Y$ is a column vector-valued function with $N$ components, $N$ is a positive integer, $\lambda$ is a complex parameter. Clearly, $Q\in L^{2}_{N\times N}(0,\pi)$ if and only if each entry of the matrix $Q$ belongs to $L^{2}(0,\pi)$.

Note that Ambarzumyan$^{\textbf{,}}$s theorem for the problem (\ref{intro1})-(\ref{intro3}) with $\rho\equiv1$ and $a=1$, for the continuous potential $Q$, has been obtained in \cite{CS}. To our best knowledge, Ambarzumyan$^{\textbf{,}}$s theorem for the problem (\ref{intro1})-(\ref{intro3}) hasn't been studied.

 It is easy to verify that $L(Q)$ is a self-adjoint eigenvalue problem. Denote the set of eigenvalues  by $\sigma(Q)=\{\lambda_{n}(Q)\}_{n=0}^{\infty}$, which can be arranged in an ascending order as (counted with multiplicity)
\begin{align}
\lambda_{0}(Q) \leq \lambda_{1}(Q) \leq \cdots  \leq\lambda_{n}(Q) \leq \cdots \longrightarrow +\infty. \nonumber
\end{align}

The main result in this paper is as follows.
\begin{theorem}\label{th}
If $\lambda_{n}(Q)=\lambda_{n}(0)$ for all $n=0,1,2\cdots$, then $Q(x)=0$ a.e. $x\in(0,\pi)$.
\end{theorem}
	
\section{Analysis of the characteristic function}
	
	In this section we analyze the characteristic function of the problem $L(Q)$, which plays a key role in the proof of Theorem \ref{th}.

Let $Y_{1}(x)=Y(x)|_{(0,\frac{\pi}{2})}$, $Y_{2}(x)=Y(\pi-x)|_{(0,\frac{\pi}{2})}$, $Q_{1}(x)=Q(x)|_{(0,\frac{\pi}{2})}$, $Q_{2}(x)=Q(\pi-x)|_{(0,\frac{\pi}{2})}$.
Then the problem $L(Q)$ can be rewritten as

\begin{align}
-Y_{1}''+Q_{1}(x)Y_{1}=\lambda Y_{1}, \quad\quad x\in \left(0,\frac{\pi}{2}\right),\label{2.1}\\
-Y_{2}''+Q_{2}(x)Y_{2}=\lambda \alpha^{2}Y_{2}, \quad x\in \left(0,\frac{\pi}{2}\right),\label{2.2}
\end{align}
with
\begin{align}
aY_{1}\left(\frac{\pi}{2}\right)-Y_{2}\left(\frac{\pi}{2}\right)=0,\quad a^{-1}Y_{1}'\left(\frac{\pi}{2}\right)+Y_{2}'\left(\frac{\pi}{2}\right)=0,\label{2.3}
\end{align}
and
\begin{align}
Y_{1}'(0)=0=Y_{2}'(0).
\end{align}

Respectively, let $\Phi_{1}(x,\lambda)$ and $\Phi_{2}(x,\lambda)$ be the solutions of the matrix differential equations
\begin{align}
-\Phi''+Q_{1}(x)\Phi=\lambda\Phi \quad \nonumber
\end{align}
and
\begin{align}
-\Phi''+Q_{2}(x)\Phi=\lambda\alpha^{2}\Phi \nonumber
\end{align}
satisfying the initial conditions
\begin{align}
\Phi(0)=I,\quad \Phi'(0)=0,\nonumber
\end{align}
where $I$ is the $N\times N$ identity matrix. Then any solution $Y_{1}(x,\lambda)$ of the equation (\ref{2.1}) satisfying $Y_{1}'(0)=0$ can be expressed as
\begin{align}
Y_{1}(x,\lambda)=\Phi_{1}(x,\lambda)C_{1}(\lambda),\label{2.4}
\end{align}
and any solution $Y_{2}(x,\lambda)$ of the equation (\ref{2.2}) satisfying $Y_{2}'(0)=0$ can also be expressed as
\begin{align}
Y_{2}(x,\lambda)=\Phi_{2}(x,\lambda)C_{2}(\lambda),\label{2.5}
\end{align}
where $C_{j}(\lambda)\; (j=1,2)$ is arbitrary $N\times 1 $ constant vector only depending on $\lambda$.

Using (\ref{2.4}) and (\ref{2.5}) in (\ref{2.3}) we obtain
\begin{align}
\begin{cases}
a\Phi_{1}(\frac{\pi}{2},\lambda)C_{1}(\lambda)-\Phi_{2}(\frac{\pi}{2},\lambda)C_{2}(\lambda)=0,\\
a^{-1}\Phi_{1}'(\frac{\pi}{2},\lambda)C_{1}(\lambda)+\Phi_{2}'(\frac{\pi}{2},\lambda)C_{2}(\lambda)=0.
\end{cases} \nonumber
\end{align}
If the above system has nonzero solutions with respect to variables $C_{1}(\lambda)$ and $C_{2}(\lambda)$, then the coefficient determinant of the system must be vanished. Thus,  the characteristic equation of the problem $L(Q)$ is given as
\begin{align}
det\left(
     \begin{array}{cc}
       a\Phi_{1}(\frac{\pi}{2},\lambda) & -\Phi_{2}(\frac{\pi}{2},\lambda) \\
       a^{-1}\Phi_{1}'(\frac{\pi}{2},\lambda) & \Phi_{2}'(\frac{\pi}{2},\lambda) \\
     \end{array}
   \right)=0. \nonumber
\end{align}
The determinant
\begin{align}
\omega_{Q}(\lambda):=det\left(
     \begin{array}{cc}
       a\Phi_{1}(\frac{\pi}{2},\lambda) & -\Phi_{2}(\frac{\pi}{2},\lambda) \\
       a^{-1}\Phi_{1}'(\frac{\pi}{2},\lambda) & \Phi_{2}'(\frac{\pi}{2},\lambda) \\
     \end{array}
     \right) \label{2.7}
\end{align}
is called the characteristic function of the problem $L(Q)$, and its zeros coincide with the eigenvalues of $L(Q)$.

We now derive the asymptotic of the characteristic function $\omega_{Q}(\lambda)$. In the following we use $M_{ii}$ to represent the entry of matrix $M$ at the $i$-st row and $i$-st column, and $trM$ to represent the trace of matrix $M$. Note that (Ref. \cite{YHY2})
\begin{align}\label{2.8}
\begin{cases}
\Phi_{1}\left(\frac{\pi}{2},\lambda\right)=\cos\Big(\frac{\sqrt{\lambda}\pi}{2}\Big)I+
\frac{\sin(\frac{\sqrt{\lambda}\pi}{2})}{\sqrt{\lambda}}[Q_{1}]
+\frac{\Psi_{1}(\lambda)}{\sqrt{\lambda}},\qquad\qquad\;\;\:\,\\
\Phi_{1}'\left(\frac{\pi}{2},\lambda\right)=-\sqrt{\lambda}\sin\Big(\frac{\sqrt{\lambda}\pi}{2}\Big)I+
\cos\Big(\frac{\sqrt{\lambda}\pi}{2}\Big)[Q_{1}]+\Psi_{2}(\lambda),\qquad\\
\Phi_{2}\left(\frac{\pi}{2},\lambda\right)=\cos\Big(\frac{\sqrt{\lambda}\pi\alpha}{2}\Big)I+
\frac{\sin(\frac{\sqrt{\lambda}\pi\alpha}{2})}{\sqrt{\lambda}\alpha}[Q_{2}]
+\frac{\Psi_{3}(\lambda)}{\sqrt{\lambda}},\qquad\quad\;\\
\Phi_{2}'\left(\frac{\pi}{2},\lambda\right)=-\sqrt{\lambda}\alpha\sin\Big(\frac{\sqrt{\lambda}\pi\alpha}{2}\Big)I+
\cos\Big(\frac{\sqrt{\lambda}\pi\alpha}{2}\Big)[Q_{2}]+\Psi_{4}(\lambda),\!\!\!
\end{cases}\!\!\!\!\!\!\!\!\!\!\!\!\!\!\!\!\!\!
\end{align}
where $\Psi_{1},\Psi_{2}\in \mathcal{L}_{N\times N}^{\frac{\pi}{2}}$, $\Psi_{3},\Psi_{4}\in \mathcal{L}_{N\times N}^{\frac{\alpha\pi}{2}}$ ( $\mathcal{L}_{N\times N}^{b}$ is the class of matrix-valued entire functions of order $\frac{1}{2}$ and exponential type $\leq b$, belonging to $L_{N\times N}^{2}(\mathbb{R})$ for real $\lambda$), $[Q_{j}]=\frac{1}{2}\int_{0}^{\frac{\pi}{2}}Q_{j}(x)dx, \; j=1,2$.

Using (\ref{2.8}) in (\ref{2.7}) and with the help of the Laplace expansion of determinants, we have
\begin{align}\label{2.10}
\omega_{Q}(\lambda)=(\sqrt{\lambda})^{N}det\left(
     \begin{array}{cc}
       a\Phi_{1}(\frac{\pi}{2},\lambda) & -\Phi_{2}(\frac{\pi}{2},\lambda) \\
       a^{-1}\frac{\Phi_{1}'(\frac{\pi}{2},\lambda)}{\sqrt{\lambda}} & \frac{\Phi_{2}'(\frac{\pi}{2},\lambda)}{\sqrt{\lambda}} \\
     \end{array}
     \right)\nonumber\qquad\qquad\quad\quad\;\\
=(\sqrt{\lambda})^{N}\Bigg\{\prod_{i=1}^{N}\Big[a\cos\frac{\sqrt{\lambda}\pi}{2}
+a\frac{\sin\frac{\sqrt{\lambda}\pi}{2}}{\sqrt{\lambda}}[Q_{1}]_{ii}
+\frac{[\Psi_{1}(\lambda)]_{ii}}{\sqrt{\lambda}}\Big]\!\nonumber\\
\times\prod_{i=1}^{N}\Big[-\alpha\sin\frac{\sqrt{\lambda}\pi\alpha}{2}
+\frac{\cos\frac{\sqrt{\lambda}\pi\alpha}{2}}{\sqrt{\lambda}}[Q_{2}]_{ii}
+\frac{[\Psi_{4}(\lambda)]_{ii}}{\sqrt{\lambda}}\Big]\quad\quad \\
+(-1)^{N}\prod_{i=1}^{N}\Big[-a^{-1}\sin\frac{\sqrt{\lambda}\pi}{2}
+a^{-1}\frac{\cos\frac{\sqrt{\lambda}\pi}{2}}{\sqrt{\lambda}}[Q_{1}]_{ii}
+\frac{[\Psi_{2}(\lambda)]_{ii}}{\sqrt{\lambda}}\Big] \!\!\!\!\!\!\!\!\!\!\!\!\nonumber\\
\times\prod_{i=1}^{N}\Big[-\cos\frac{\sqrt{\lambda}\pi\alpha}{2}
-\frac{\sin\frac{\sqrt{\lambda}\pi\alpha}{2}}{\sqrt{\lambda}\alpha}[Q_{2}]_{ii}
+\frac{[\Psi_{3}(\lambda)]_{ii}}{\sqrt{\lambda}}\Big]+\frac{\psi(\lambda)}{\lambda}\Bigg\}\!\!\!\!\!\!\!\!\nonumber\\
:=\omega_{0}(\lambda)+(\sqrt{\lambda})^{N-1}\big[G(\lambda)+\psi(\lambda)\big], \qquad\qquad\qquad\qquad\quad\;\nonumber
\end{align}
where $\psi\in \mathcal{L}^{\frac{N\pi(1+\alpha)}{2}}$ ($\mathcal{L}^{b}$ is the class of scalar entire functions of order $\frac{1}{2}$ and exponential type $\leq b$, belonging to $L^{2}(\mathbb{R})$ for real $\lambda$),
\begin{align}
\omega_{0}(\lambda)=(\sqrt{\lambda})^{N}\Big[\left(-\frac{\alpha a}{2}\right)^{N}
\Big(\sin\frac{\sqrt{\lambda}\pi(\alpha+1)}{2}+\sin\frac{\sqrt{\lambda}\pi(\alpha-1)}{2}\Big)^{N}\;\nonumber\\
+\left(-\frac{a^{-1}}{2}\right)^{N}\Big(\sin\frac{\sqrt{\lambda}\pi(\alpha+1)}{2}
-\sin\frac{\sqrt{\lambda}\pi(\alpha-1)}{2}\Big)^{N}\Big],\quad\quad\;\:\nonumber\\
G(\lambda)=\small{A(\lambda)\left(-\frac{\alpha a}{2}\right)^{N-1}
\Big(\sin\frac{\sqrt{\lambda}\pi(\alpha+1)}{2}+\sin\frac{\sqrt{\lambda}\pi(\alpha-1)}{2}\Big)^{N-1}}
\!\!\nonumber\\
+\small{B(\lambda)\left(-\frac{a^{-1}}{2}\right)^{N-1}\Big(\sin\frac{\sqrt{\lambda}\pi(\alpha+1)}{2}
-\sin\frac{\sqrt{\lambda}\pi(\alpha-1)}{2}\Big)^{N-1}}\!\!\!\!\!\!\,\nonumber
\end{align}
with
\begin{align}
A(\lambda)=\frac{a\, tr[Q_{2}]+\alpha a\, tr[Q_{1}]}{2}\cos\frac{\sqrt{\lambda}\pi(\alpha+1)}{2}\qquad\qquad\nonumber\\
+\frac{a\, tr[Q_{2}]-\alpha a\, tr[Q_{1}]}{2}\cos\frac{\sqrt{\lambda}\pi(\alpha-1)}{2},\nonumber\qquad\qquad\\
B(\lambda)=\frac{\alpha a^{-1}tr[Q_{1}]+a^{-1}tr[Q_{2}]}{2\alpha}\cos\frac{\sqrt{\lambda}\pi(\alpha+1)}{2}\qquad\;\nonumber\\
+\frac{\alpha a^{-1}tr[Q_{1}]-a^{-1}tr[Q_{2}]}{2\alpha}\cos\frac{\sqrt{\lambda}\pi(\alpha-1)}{2}.\qquad\;\nonumber
\end{align}
In fact, the first term $\omega_{0}(\lambda)$ is the characteristic function of the problem $L(0)$. By the Palay-Wiener$^{\textbf{,}}$s theorem we find that for $\lambda\in\mathbb{R}$,
\begin{align}
\psi(\lambda)=o(1),\quad \lambda\rightarrow+\infty,\nonumber
\end{align}
and for $\sqrt{\lambda}=i\kappa(\kappa>0)$,
\begin{align}\label{2.10.1}
\psi(\lambda)=o\left(e^{\frac{\kappa N\pi(1+\alpha)}{2}}\right),\quad \kappa\rightarrow+\infty.
\end{align}

\section{Proof}
	
	In this section we give the proof of Theorem \ref{th}, for which we need the following lemmas.

\begin{lemma}\label{lem1}
The characteristic function $\omega_{Q}(\lambda)$ of the problem $L(Q)$ is uniquely determined by $\sigma(Q)=\{\lambda_{n}(Q)\}_{n=0}^{\infty}$ (counted with multiplicity).
\end{lemma}

\begin{proof}
From (\ref{2.10}), we observe that $\omega_{Q}(\lambda)$ is entire in $\lambda$ of order $\frac{1}{2}$. Hence, by Hadamard$^{\textbf{,}}$s factorization theorem,  $\omega_{Q}(\lambda)$ is uniquely determined up to a multiplicative constant $C(Q)$ by its zeros:
\begin{align}\label{3.1}
\omega_{Q}(\lambda)=C(Q)\lambda^{m}\prod_{\{n:\lambda_{n}(Q)\neq0\}}\left(1-\frac{\lambda}{\lambda_{n}(Q)}\right),
\end{align}
where $m$ is the multiplicity of zero eigenvalues (If zero isn't an eigenvalue, then $m=0$). Substituting $\sqrt{\lambda}=i\kappa(\kappa>0)$ into (\ref{2.10}) yields that as $\kappa\rightarrow+\infty$,
\begin{align}
\omega_{Q}(-\kappa^{2})=\omega_{0}(-\kappa^{2})+(i\kappa)^{N-1}\left[G(-\kappa^{2})+\psi(-\kappa^{2})\right]\quad\;\;\;\nonumber\\
=\kappa^{N}e^{\frac{\kappa N\pi(1+\alpha)}{2}}\left[\left(\frac{\alpha a}{4}\right)^{N}+\left(\frac{a^{-1}}{4}\right)^{N}\right]\big[1+o(1)\big],\!\!\!\!\!\!\label{3.1.1}
\end{align}
where $G(-\kappa^{2})=O\left(e^{\frac{\kappa N\pi(1+\alpha)}{2}}\right)$ and (\ref{2.10.1}) are used.

Hence, from (\ref{3.1}) and (\ref{3.1.1}), we obtain
\begin{align}\label{3.2}
C(Q)=(-1)^{m}\left[\left(\frac{\alpha a}{4}\right)^{N}+\left(\frac{a^{-1}}{4}\right)^{N}\right]
\qquad\qquad\qquad\qquad\qquad\;\;\,\nonumber\\
\small{\times\lim_{\kappa\rightarrow+\infty}\Bigg\{\kappa^{N-2m}e^{\frac{\kappa N\pi(1+\alpha)}{2}}
\Bigg[\prod_{\{n:\lambda_{n}(Q)\neq0\}}\left(1+\frac{\kappa^{2}}{\lambda_{n}(Q)}\right)\Bigg]^{-1}\Bigg\}}.\!\!\!\!\!\!\!\!\!\!\!
\end{align}
With the help of (\ref{3.1}) and (\ref{3.2}), the proof of Lemma \ref{lem1} is complete.
\end{proof}

\begin{lemma}\label{lem2}
If $\lambda_{n}(Q)=\lambda_{n}(0)$ for all $n=0,1,2\cdots$, then
\begin{align}
tr[Q_{1}]=tr[Q_{2}]=0.\nonumber
\end{align}
\end{lemma}

\begin{proof}
Since $\lambda_{n}(Q)=\lambda_{n}(0)$ for all $n=0,1,2\cdots$, Lemma \ref{lem1} tells that
$\omega_{Q}(\lambda)=\omega_{0}(\lambda)$. Thus, it follows from (\ref{2.10}) that for $\lambda\in\mathbb{C}$,
\begin{align}
G(\lambda)+\psi(\lambda)=0,  \nonumber
\end{align}
that is,
\begin{align}\label{3.3}
A(\lambda)\left(-\frac{\alpha a}{2}\right)^{N-1}
\Big(\sin\frac{\sqrt{\lambda}\pi(\alpha+1)}{2}+\sin\frac{\sqrt{\lambda}\pi(\alpha-1)}{2}\Big)^{N-1}
\nonumber\\
+B(\lambda)\left(-\frac{a^{-1}}{2}\right)^{N-1}\Big(\sin\frac{\sqrt{\lambda}\pi(\alpha+1)}{2}
-\sin\frac{\sqrt{\lambda}\pi(\alpha-1)}{2}\Big)^{N-1} \!\!\!\!\!\!\!\!\! \\
+\psi(\lambda)=0.\nonumber\qquad\qquad\qquad\qquad\qquad\qquad\qquad\qquad\qquad\qquad\;\:
\end{align}

In particular, taking $\sqrt{\lambda}=2n$ ($n\in\mathbb{N}$), then the equation (\ref{3.3}) is transformed into
\begin{align}\label{3.4}
a\,tr[Q_{2}]\cos(n\pi\alpha)\Big(2\sin(n\pi\alpha)\Big)^{N-1}+o(1)=0,\quad n\rightarrow+\infty.
\end{align}

When the impulse parameter $\alpha\in(0,1)$ is a irrational number, then (\ref{3.4}) deduces $tr[Q_{2}]=0$.

When the impulse parameter $\alpha\in(0,1)$ is a rational number, denoting $\alpha=\frac{q}{p}$ ( $p$, $q$ are coprime natural number and $q<p$ ). Moreover, if $\alpha=\frac{q}{p}\neq\frac{1}{2}$, then taking $n=pk+1$ ($k\in\mathbb{N}$) in (\ref{3.4}) yields
\begin{align}
a\,tr[Q_{2}]\cos\frac{q\pi}{p}\Big(2\sin\frac{q\pi}{p}\Big)^{N-1}+o(1)=0,\quad k\rightarrow+\infty,\nonumber
\end{align}
which deduces $tr[Q_{2}]=0$.

If $\alpha=\frac{q}{p}=\frac{1}{2}$, then taking $\sqrt{\lambda}=8n+1$ ($n\in\mathbb{N}$), the equation (\ref{3.3}) is transformed into
\begin{align}
a^{-1}tr[Q_{2}]+o(1)=0,\quad n\rightarrow+\infty,\nonumber
\end{align}
which yields $tr[Q_{2}]=0$.

Substituting $tr[Q_{2}]=0$ into the equation (\ref{3.3}), we obtain
\begin{align}
tr[Q_{1}]\Bigg\{(\alpha a)^{N}\Big(\sin\frac{\sqrt{\lambda}\pi\alpha}{2}\Big)^{N}
\Big(\cos\frac{\sqrt{\lambda}\pi}{2}\Big)^{N-1}\sin\frac{\sqrt{\lambda}\pi}{2}\qquad\;\;\; \nonumber \\
-a^{-N}\Big(\cos\frac{\sqrt{\lambda}\pi\alpha}{2}\Big)^{N}
\Big(\sin\frac{\sqrt{\lambda}\pi}{2}\Big)^{N-1}\cos\frac{\sqrt{\lambda}\pi}{2}\Bigg\}
+\psi(\lambda)=0, \nonumber
\end{align}
which is equivalent to
\begin{align}\label{3.4.1}
tr[Q_{1}]\Bigg\{(\alpha a)^{N}\frac{\Big(\sin\frac{\sqrt{\lambda}\pi\alpha}{2}\Big)^{N}
\Big(\cos\frac{\sqrt{\lambda}\pi}{2}\Big)^{N-1}\sin\frac{\sqrt{\lambda}\pi}{2}}{e^{\frac{\tau N\pi(1+\alpha)}{2}}}\nonumber\qquad\qquad\quad \\
-a^{-N}\frac{\Big(\cos\frac{\sqrt{\lambda}\pi\alpha}{2}\Big)^{N}
\Big(\sin\frac{\sqrt{\lambda}\pi}{2}\Big)^{N-1}\cos\frac{\sqrt{\lambda}\pi}{2}}{e^{\frac{\tau N\pi(1+\alpha)}{2}}}\Bigg\}
+\frac{\psi(\lambda)}{e^{\frac{\tau N\pi(1+\alpha)}{2}}}=0,
\end{align}
where $\tau=|Im\sqrt{\lambda}|$.

In particular, taking $\sqrt{\lambda}=i\kappa$ ($\kappa>0$) in (\ref{3.4.1}) and with the help of (\ref{2.10.1}), we get as $\kappa\rightarrow+\infty$,
\begin{align}
tr[Q_{1}]\Bigg\{(\alpha a)^{N}\left(\frac{e^{-\kappa\pi\alpha}-1}{2i}\right)^{N}
\left(\frac{e^{-\kappa\pi}+1}{2}\right)^{N-1}\left(\frac{e^{-\kappa\pi}-1}{2i}\right)\qquad\;\;\nonumber \\
-a^{-N}\left(\frac{e^{-\kappa\pi\alpha}+1}{2}\right)^{N}
\left(\frac{e^{-\kappa\pi}-1}{2i}\right)^{N-1}\left(\frac{e^{-\kappa\pi}+1}{2}\right)\Bigg\}+o(1)=0.\!\nonumber
\end{align}
Letting $\kappa\rightarrow+\infty$ in the above equation yields
\begin{align}
tr[Q_{1}]\left[(\alpha a)^{N}+a^{-N}\right]=0,\nonumber
\end{align}
that is,
\begin{align}
tr[Q_{1}]=0.\nonumber
\end{align}
The proof of Lemma \ref{lem2} is complete.
\end{proof}

Introducing the Hilbert space $L_{N}^{2}(0,\pi):=\oplus_{i=1}^{N}L^{2}(0,\pi)$ with the inner product
\begin{align}
(f,g)=\int_{0}^{\pi}g^{\dag}(x)f(x)dx=\sum_{i=1}^{N}\int_{0}^{\pi}f_{i}(x)\overline{g_{i}(x)}dx,\nonumber
\end{align}
where $f=(f_{1},\cdots,f_{N})^{t}\in L_{N}^{2}(0,\pi)$, $g=(g_{1},\cdots,g_{N})^{t}\in L_{N}^{2}(0,\pi)$, and $g^{t}$ denotes the transpose of the vector $g$, and $g^{\dag}$ denotes the conjugate transpose of the vector $g$. Clearly, $f=(f_{1},\cdots,f_{N})^{t}\in L_{N}^{2}(0,\pi)$ if and only if $f_{i}\in L^{2}(0,\pi),\;i=1,\cdots,N$. The domain of self-adjoint operator $L(Q)$ is
\begin{align}
D(L(Q))=\Big\{Y\in L_{N}^{2}(0,\pi)\:|\:\small{Y'\in AC_{N,\,loc}\left(\Big(0,\frac{\pi}{2}\Big)\cup\Big(\frac{\pi}{2},\pi\Big)\right)}, \qquad\;\;\nonumber\\
\:lY\in L_{N}^{2}(0,\pi), \;Y satisfying \;(\ref{intro2}) \;and \;(\ref{intro3})\Big\}.\nonumber\!\!\!\!\!
\end{align}

Now we give the proof of Theorem \ref{th}.
\\

 \textit{Proof of Theorem \ref{th}}.  It is easy to see that the operator $L(0)$ is non-negative and $0\in\sigma(0)$, so zero is its smallest eigenvalue, namely $\lambda_{0}(0)=0$.

Let
\begin{align}
 Y_{i}(x)=\begin{cases}
          e_{i},\:\:\,\, 0<x<\frac{\pi}{2},\\
          ae_{i},\:\,\frac{\pi}{2}<x<\pi,
        \end{cases}\nonumber
\end{align}
where $e_{i}$ is the $N\times1$ unit vector, whose the $i$-st component is $1\;(i=1,2,\cdots,N)$.

By the variational principle, we obtain
\begin{align}\label{3.5}
0=\lambda_{0}(0)=\lambda_{0}(Q)\qquad\qquad\qquad\qquad\qquad\qquad\qquad\nonumber\\
=\inf_{0\neq Y\in D(L(Q))}\frac{\int_{0}^{\pi}\left[-Y^{\dag}(x)Y''(x)+Y^{\dag}(x)Q(x)Y(x)\right]dx}
{\int_{0}^{\pi}\left[\rho(x)Y^{\dag}(x)Y(x)\right]dx}\!\!\!\!\!\!\!\!\!\!\!\!\!\!\!\!\!\nonumber\\
\leq\frac{\int_{0}^{\pi}\big[-Y^{\dag}_{i}(x)Y''_{i}(x)+Y^{\dag}_{i}(x)Q(x)Y_{i}(x)\big]dx}
{\int_{0}^{\pi}\big[\rho(x)Y^{\dag}_{i}(x)Y_{i}(x)\big]dx}\qquad\\
=\frac{\int_{0}^{\frac{\pi}{2}}Q_{1ii}(x)dx+a^{2}\int_{0}^{\frac{\pi}{2}}Q_{2ii}(x)dx}
{\frac{\pi}{2}(1+\alpha^{2}a)}\qquad\qquad\qquad\nonumber\\
=\frac{2[Q_{1}]_{ii}+2a^{2}[Q_{2}]_{ii}}{\frac{\pi}{2}(1+\alpha^{2}a)},\qquad i=1,2,\cdots,N.\nonumber \qquad\quad
\end{align}
Note that $M_{ii}$ represents the entry of matrix $M$ at the $i$-st row and $i$-st column. Adding $N$ inequalities in (\ref{3.5}), together with Lemma \ref{lem2}, yields that
\begin{align}
0\leq\frac{2tr[Q_{1}]+2a^{2}tr[Q_{2}]}{\frac{\pi}{2}(1+\alpha^{2}a)}=0.\nonumber
\end{align}
Therefore each of right-hand side in (\ref{3.5}) vanishes. So $Y_{i}(x)\;(i=1,2,\cdots,N)$ is the eigenfunction corresponding to the first eigenvalue 0 of $L(Q)$.

Substituting $Y_{i}(x)$ into equation (\ref{intro1}), we have
\begin{align}
Q(x)e_{i}=0 \; a.e. \;x\in(0,\pi),\quad i=1,2,\cdots N, \nonumber
\end{align}
which are equivalent to
\begin{align}
Q(x)I=0 \; a.e. \;x\in(0,\pi). \nonumber
\end{align}
Thus $Q(x)=0$ a.e. $x\in(0,\pi)$. The proof of Theorem \ref{th} is complete. $\;\;\;\;\:\square$

\qquad
	
	\noindent {\bf Acknowledgments.}
	This work was supported in part by  the National Natural Science Foundation of China (11871031) and the National Natural Science Foundation of Jiang Su (BK20201303).

\end{document}